\documentclass[11pt,leqno]{article}

\usepackage{amsmath,amssymb,amsthm}
\usepackage[all]{xy}
\usepackage{lscape}

\makeatletter

\newtheorem*{Theorem}{Theorem~\ref{main.thm}}

\newtheorem{theorem}{Theorem}[section]

\newtheorem{lemma}[theorem]{Lemma}

\theoremstyle{definition}

\newtheorem{example}[theorem]{Example}
\newtheorem{definition}[theorem]{Definition}

\theoremstyle{remark}

\theoremstyle{remark}
\newtheorem{remark}[theorem]{Remark}

\def\({{\rm (}}
\def\){{\rm )}}

\let\Mathrm\operator@font
\let\Cal\mathcal
\let\Bbb\mathbb

\def\standop#1{\mathop{\Mathrm #1}\nolimits}
\def\difstop#1#2{\expandafter\def\csname #1\endcsname{\standop{#2}}}
\def\defstop#1{\difstop{#1}{#1}}

\defstop{AB}
\defstop{ann}
\defstop{Ass}
\defstop{Add}
\defstop{Alt}
\defstop{Ass}
\difstop{adim}{AFHdim}

\defstop{Bl}

\defstop{CMFI}
\defstop{codim}
\defstop{Coh}
\defstop{coht}
\defstop{Coker}
\defstop{Cone}
\defstop{Cl}
\defstop{Cox}
\defstop{cosk}
\defstop{cd}
\defstop{cmd}
\difstop{charac}{char}

\defstop{depth}
\defstop{Der}
\defstop{Div}
\defstop{div}

\defstop{EM}
\defstop{embdim}
\defstop{End}
\defstop{ev}
\defstop{Ext}

\difstop{fdim}{flat.dim}
\defstop{Flat}
\defstop{Func}
\defstop{Fpqc}

\def\GL{\text{\sl{GL}}}
\defstop{GD}
\defstop{Good}
\defstop{Gal}
\defstop{Grass}

\difstop{height}{ht}
\defstop{Hom}
\defstop{Hy}

\def\id{\mathord{\Mathrm{id}}}

\difstop{Image}{Im}
\defstop{ind}
\defstop{ini}

\defstop{Ker}

\defstop{Lch}
\defstop{length}
\defstop{Lin}
\defstop{Lqc}
\defstop{lqc}
\defstop{LQ}
\defstop{LM}
\defstop{Lie}
\defstop{Loc}

\defstop{Mat}
\defstop{Max}
\defstop{Min}
\defstop{Mod}
\defstop{Mor}
\defstop{MCM}
\defstop{Map}
\difstop{mred}{red}

\defstop{Nerve}
\defstop{NonSFR}
\defstop{NonCMFI}
\defstop{NonNor}
\defstop{Nor}

\defstop{ob}
\defstop{Ob}

\defstop{PA}
\defstop{PM}
\defstop{PR}
\defstop{Proj}
\defstop{Prin}
\defstop{Pic}

\defstop{Qch}
\defstop{qch}

\defstop{rad}
\defstop{rank}
\def\red{_{\Mathrm{red}}}

\defstop{res}
\defstop{Reg}
\difstop{Refl}{Ref}
\defstop{Rat}

\def\SL{\text{\sl{SL}}}

\defstop{Spec}
\defstop{supp}
\defstop{Supp}
\defstop{Sym}
\defstop{Sing}
\defstop{SFR}
\defstop{Soc}
\defstop{Sh}

\difstop{tdeg}{trans.deg}

\defstop{Tor}
\defstop{Tr}
\defstop{TRD}
\difstop{trace}{tr}
\defstop{tot}

\defstop{Zar}
\defstop{FL}
\defstop{add}
\defstop{Ind}
\defstop{grmod}
\defstop{soc}
\def\ad{_{\mathrm{ad}}}
\difstop{summ}{sum}

\def\gl{\mathop{\mathfrak{gl}}\nolimits}

\let\frak\mathfrak
\def\({{\rm(}}
\def\){{\rm)}}

\def\O{\Cal O}

\def\ZZ{{\mathbb{Z}}}

\def\II{{\mathbb{I}}}

\def\sdarrow#1{\downarrow\hbox to 0pt{\scriptsize$#1$\hss}}
\def\suarrow#1{\uparrow\hbox to 0pt{\scriptsize$#1$\hss}}
\def\ssearrow#1{\searrow\hbox to 0pt{\scriptsize$#1$\hss}}

\def\section{\@startsection{section}{1}{\z@ }%
  {-3.5ex plus -1ex minus -.2ex}{2.3ex plus .2ex}{\bf }}

\long\def\refname{\par\kern -3ex
  \begin{center}\rm R\sc{eferences}\end{center}\par\kern 
  -2ex}

\def\@seccntformat#1{\csname the#1\endcsname.\quad}

\def\@@@sect#1#2#3#4#5#6[#7]#8{%
  \ifnum #2>\c@secnumdepth 
  \def \@svsec {}\else \refstepcounter {#1}%
  \def\@svsec{}
  \fi 
  \@tempskipa #5\relax 
  \ifdim \@tempskipa >\z@ 
  \begingroup #6\relax \@hangfrom {\hskip #3\relax 
    \@svsec}{\interlinepenalty \@M #8\par }\endgroup 
  \csname #1mark\endcsname {#7}
  \else 
  \def \@svsechd {#6\hskip #3\@svsec #8\csname #1mark\endcsname {#7}}
  \fi \@xsect {#5}}

\def\@@@startsection#1#2#3#4#5#6{%
  \if@noskipsec \leavevmode \fi \par \@tempskipa #4\relax \@afterindenttrue 
  \ifdim \@tempskipa <\z@ \@tempskipa -\@tempskipa \@afterindentfalse 
  \fi \if@nobreak \everypar {}\else \addpenalty {\@secpenalty }\addvspace 
  {\@tempskipa }\fi \@ifstar {\@ssect {#3}{#4}{#5}{#6}}{\@dblarg 
    {\@@@sect {#1}{#2}{#3}{#4}{#5}{#6}}}}

\def\theparagraph{\thesection.\arabic{paragraph}}
\def\aparagraph{\@@@startsection{paragraph}{2}{\z@ }%
  {1.75ex plus .2ex minus .15ex}{-1em}{\bf(\theparagraph) } }
\def\paragraph{\@@@startsection{paragraph}{2}{\z@ }%
  {1.75ex plus .2ex minus .15ex}{-1em}{}{\bf(\theparagraph)} }

\let\c@theorem\c@paragraph

\advance\textheight 16mm
\advance\textwidth 20mm
\title{The symmetry of finite group schemes, Watanabe type theorem, and the $a$-invariant of the ring of invariants}
\author{M{\sc itsuyasu} H{\sc ashimoto}\thanks{Partially supported by JSPS KAKENHI Grant number 20K03538
    and MEXT Promotion of Distinctive Joint Research Center Program JPMXP0723833165.}}
\date{Dedicated to Professor Mel Hochster and Professor Craig Huneke}
  
\begin{document}

\maketitle

\begin{abstract}
  Let $k$ be a field, and $G$ be a $k$-group scheme of finite type.
  Let $G\ad$ be the $k$-scheme $G$ with the adjoint action of $G$.
  We call $\lambda_{G,G}=H^0(\Spec k,e^*(\omega_{G\ad}))$ the Knop character of $G$,
  where $e:\Spec k\rightarrow G\ad$ is the unit element, and $\omega_{G\ad}$ is the
  $G$-canonical module.
  We prove that $\lambda_{G,G}$ is trivial in the following cases:
  (1) $G$ is finite, and $k[G]^*$ is a symmetric algebra;
  (2) $G$ is finite and \'etale;
  (3) $G$ is finite and constant;
  (4) $G$ is smooth and connected reductive;
  (5) $G$ is abelian;
  (6) $G$ is finite, and the identity component $G^\circ$ of $G$ is linearly reductive;
  (7) $G$ is finite and linearly reductive.
  Let $V$ be a small $G$-module of dimension $n<\infty$.
  We assume that $\lambda_{G,G}$ is trivial.
  Let $H=\Bbb G_m$ be the one-dimensional torus, and let $V$ be of degree one as an $H$-module
  so that $S=\Sym V^*$ is a $\tilde G$-algebra generated by degree one elements,
  where $\tilde G=G\times H$.
  We set $A=S^G$.
  Then we have
    (i) $\omega_A\cong\omega_S^G$ as $(H,A)$-modules;
    (ii) $a(A)\leq -n$ in general, where $a(A)$ denotes the $a$-invariant.
  Moreover, the following are equivalent:
  (1) The action $G\rightarrow\GL(V)$ factors through $\SL(V)$;
  (2) $\omega_S\cong S(-n)$ as $(\tilde G,S)$-modules;
  (3) $\omega_S\cong S$ as $(G,S)$-modules;
  (4) $\omega_A\cong A(-n)$ as $(H,A)$-modules;
  (5) $A$ is quasi-Gorenstein;
  (6) $A$ is quasi-Gorenstein and $a(A)=-n$;
  (7) $a(A)=-n$.
  This partly generalizes recent results of Liedtke--Yasuda   {\tt arXiv:2304.14711v2}
  and
  Goel--Jeffries--Singh {\tt arXiv:2306.14279v1}.
\end{abstract}

\section{Introduction}

Let $k$ be a field, $V$ a finite-dimensional $k$-vector space,
and $G$ a finite subgroup of $\GL(V)$.
Let $S=\Sym V^*=k[V]$, and $A=S^G$.
Hochster and Eagon proved that in non-modular case (that is, the case that
the order $|G|$ of $G$ is not divisible by the characteristic of $k$),
$A$ is Cohen--Macaulay.
K.-i.~Watanabe proved that in non-modular case, $G\subset\SL(V)$ if and only if
$G$ does not have a pseudo-reflection and $A$ is Gorenstein \cite{Wat1,Wat2}.
Since then, his result has been generalized by several authors.
Fleischmann and Woodcock \cite{FW} and Braun \cite{Braun} proved that if $G\subset \GL(V)$ is
a finite subgroup without pseudo-reflection, then $A$ is quasi-Gorenstein (or equivalently,
$\omega_A\cong A$) if and only if $G\subset\SL(V)$.

It has been known that the condition that the finite group $G$ does not have a pseudo-reflection
sometimes can be generalized to more general $G$.
The condition is replaced by the condition that $\pi:V=\Spec S\rightarrow\Spec A=V/\!/G$ is a
principal $G$-bundle off codimension two or more, and called an almost principal bundle
or quasi-torsor \cite{Has4,CR}, and we call this condition \lq $V$ is small.'
Namely, we say that $V$ is small
if there exist some open subset $W$ of $\Spec A$ and $G$-stable open subset $U$ of $\pi^{-1}(W)$
such that $\codim(V\setminus U,V)\geq 2$, $\codim(V/\!/G\setminus W,V/\!/G)\geq 2$, and
$\pi:U\rightarrow W$ is a principal $G$-bundle (or a $G$-torsor) in the sense that
$\pi$ is faithfully flat, and $\Phi:G\times U\rightarrow U\times_W U$ given by $\Phi(g,u)=(gu,u)$
is an isomorphism.
Note that if $G$ is a finite constant group, then $V$ is small if and only if $G\subset\GL(V)$,
and $G$ does not have a pseudo-reflection.

Knop \cite{Knop} pointed out that
the equivalence $\omega_A\cong A\iff G\subset\SL(V)$
is not true any more even if $G$ is a (disconnected) reductive
group over an algebraically closed field of characteristic zero, and the action is small.
Letting $\lambda\ad$ be the top exterior power of $\Lie(G)^*$, the dual of the adjoint representation,
the triviality of $\det_V\otimes\lambda\ad^*$ was important \cite[Satz~2]{Knop}.
Note that $\lambda\ad$ is trivial if $G$ is finite, and we can recover Watanabe's original result.

We define $\lambda_{G,G}=H^0(\Spec k,e^*(\omega_{G\ad}))$, and call it the Knop character of $G$,
where $e:\Spec k\rightarrow G\ad$ is the unit element, and $\omega_{G\ad}$ is the $G$-equivariant canonical module of $G\ad$.
If, moreover, $G$ is a normal closed subgroup scheme of another affine $k$-group scheme $\tilde G$
of finite type, then $\lambda_{G,G}$ is a character of $\tilde G$, and we denote it by $\lambda_{\tilde G,G}$.
Note that $\lambda_{G,G}\cong\lambda\ad^*$, if $G$ is $k$-smooth.
By \cite[(11.22)]{Has4}, it is easy to see that if $V$ is small, then $\omega_A\cong A(a)$ if
and only if $\omega_S=S\otimes\det_{V^*}\cong S(a)\otimes_k \lambda_{\tilde G,G}$ if and only if
$\det_V\cong\lambda_{G,G}$ as $G$-modules, and $a=-n$.
If, moreover, $\lambda_{G,G}\cong k$, then $A$ is quasi-Gorenstein if and only if $G\subset\SL(V)$, and
if these conditions are satisfied, then $a(A)=-n$.
So it is natural to ask, when $\lambda_{G,G}$ is trivial.
In \cite[(11.21)]{Has4}, it is pointed out that if $G$ is finite and linearly reductive,
\'etale, or connected reductive, then $\lambda_{G,G}$ is trivial, but $\lambda_{G,G}$ is nontrivial if $k$ is a field of
characteristic not two and $G=O(2)$.

In this paper, we discuss when $\lambda_{G,G}$ is trivial, assuming that $G$ is finite (but not \'etale).
It is well-known that the group algebra $kG$ is symmetric \cite[Example~IV.2.6]{SY}.
A finite dimensional $k$-Hopf algebra is Frobenius in general \cite[Theorem~VI.3.6]{SY}.
In general, a finite dimensional $k$-Hopf algebra $H$ is not symmetric even if $H$ is cocommutative,
or equivalently, $H=k[G]^*$ for some finite $k$-group scheme $G$, see \cite[p.~85]{LS}.
We prove that $\lambda_{G,G}$ is trivial if and only if the notions of the left integral and the
right integral agree in $H=k[G]^*$.
The latter condition is called the unimodular property of $H$.
As the square $s_H^2$ of the antipode $s_H$ of $H$ is the identity, $H$ is unimodular
if and only if $H$ is a symmetric algebra, see \cite{Humphreys, Radford}.

As an application of the $G$-triviality of $\lambda_{G,G}$, 
we prove the following.

\begin{Theorem}
  Let $k$ be a field, $G$ be an affine $k$-group scheme of finite type, and $V$ be a small $G$-module
  of dimension $n<\infty$.
  We assume that $\lambda_{G,G}$ is trivial.
  Let $H=\Bbb G_m$ be the one-dimensional torus, and let $V$ be of degree one as an $H$-module
  so that $S=\Sym V^*$ is a $\tilde G$-algebra generated by degree one elements,
  where $\tilde G=G\times H$.
  We set $A=S^G$.
  Then we have
  \begin{enumerate}
    \item[\rm(i)] $\omega_A\cong\omega_S^G$ as $(H,A)$-modules;
    \item[\rm(ii)] $a(A)\leq -n$ in general, where $a(A)$ denotes the $a$-invariant.
  \end{enumerate}
  Moreover, the following are equivalent:
  \begin{enumerate}
  \item[\rm(1)] The action $G\rightarrow\GL(V)$ factors through $\SL(V)$;
  \item[\rm(2)] $\omega_S\cong S(-n)$ as $(\tilde G,S)$-modules;
  \item[\rm(3)] $\omega_S\cong S$ as $(G,S)$-modules;
  \item[\rm(4)] $\omega_A\cong A(-n)$ as $(H,A)$-modules;
  \item[\rm(5)] $A$ is quasi-Gorenstein;
  \item[\rm(6)] $A$ is quasi-Gorenstein and $a(A)=-n$;
  \item[\rm(7)] $a(A)=-n$.
  \end{enumerate}
\end{Theorem}

For the case that $G$ is finite and constant, the theorem was proved (in a stronger form) by Goel, Jeffries, and Singh \cite{GJS}.
Note that they do not require that the action of $G$ on $V$ is small.
They proved that $a(A)\leq a(S)=-n$ in general.
They also proved that the equality $a(A)=-n$ holds if and only if the image of
$G\rightarrow\GL(V)$ is a subgroup of $\SL(V)$ without pseudo-reflections 
for the case that $G$ is finite and constant \cite[Proposition~4.1, Theorem~4.4]{GJS}.
It is interesting to ask if these are true for any finite group scheme $G$.
Note also that the equivalence (1)$\Leftrightarrow$(5) for the case that
$G$ is finite linearly reductive (but not necessarily
constant) was proved recently by Liedtke and Yasuda \cite{LY}.
It also follows from \cite[(7.61),(11.22){\bf 2}]{Has4}.

\medskip

Acknowledgment.
The author thanks K.~Goel, A.~Singh, K.-i.~Watanabe, and T.~Yasuda for valuable communications.
He is also grateful to A.~Masuoka for valuable advice.

\section{Preliminaries}

\paragraph
Let $k$ be a field, $f:\tilde G\rightarrow H$ be a homomorphism between
affine $k$-group schemes of finite type with $G=\Ker f$.
Let $\Cal F(\tilde G)$ be the category of $\tilde G$-schemes separated of finite type over $k$.
For $(h_Z:Z\rightarrow \Spec k)\in\Cal F(\tilde G)$, {\em the} $\tilde G$-dualizing complex of $Z$
(or better, of $h_Z$) is $h_Z^!(k)$ by definition, and we denote it by $\II_Z=\II_Z(G)$,
where $(-)^!$ denotes the twisted inverse \cite{Has2}.
The $\tilde G$-canonical module $\omega_Z$ is the lowest nonzero cohomology group of $\II_Z$.
It is a coherent $\tilde G$-module.
If $Z=\Spec B$ is affine, $H^0(Z,\omega_Z)$ is denoted by $\omega_B$, and is called the
$\tilde G$-equivariant canonical module of $B$.
When we forget the $\tilde G$-structure, $\II_Z$ is the dualizing complex of the scheme $Z$
without the $\tilde G$-action \cite[(31.17)]{Has2}.

\paragraph
A morphism $\varphi:X\rightarrow Y$ of $\tilde G$-schemes of finite type over $k$ is called a $\tilde G$-enriched principal $G$-bundle
if $G$ acts trivially on $Y$, $\varphi$ is faithfully flat, and the morphism $\Phi:G\times X\rightarrow X\times_YX$
given by $\Phi(g,x)=(gx,x)$ is an isomorphism.
As $G$ is affine, flat, and Gorenstein over $\Spec k$, $\varphi$ is affine, flat, and Gorenstein.

\paragraph
Let $X$ be a scheme and $U$ its open subset.
We say that $U$ is $n$-large if $\codim(X\setminus U,X)\geq n+1$,
where we regard that the codimension of the empty set in $X$ is $\infty\geq n+1$.

\begin{definition}[{cf.~\cite[(10.2)]{Has4}}]
A diagram of $\tilde G$-schemes of finite type
\begin{equation}\label{almost-intro.eq}
\xymatrix{
X & U \ar@{_{(}->}[l]_i \ar[r]^\rho & V \ar@{^{(}->}[r]^j & Y
}
\end{equation}
is called a $\tilde G$-enriched $n$-almost rational principal $G$-bundle if
(1) $G$ acts trivially on $Y$;
(2) $j$ is an open immersion, and $j(V)$ is $n$-large in $Y$;
(3) $i$ is an open immersion, and $i(U)$ is $n$-large in $X$;
(4) $\rho:U\rightarrow V$ is a principal $G$-bundle.
That is, $\rho$ is faithfully flat, and $\Phi:G\times U\rightarrow U\times_VU$
given by $\Phi(n,u)=(nu,u)$ is an isomorphism.
\end{definition}

\paragraph
In what follows, $1$-large and $1$-almost will simply be called large and almost,
respectively.
A $\tilde G$-morphism $\varphi:X\rightarrow Y$ is said to be a $\tilde G$-enriched $n$-almost
principal $G$-bundle with respect to $U$ and $V$, if $U$ is a $\tilde G$-stable open subset of $X$,
$V$ is an $H$-stable open subset of $Y$, and the diagram (\ref{almost-intro.eq}) is a
$\tilde G$-enriched $n$-almost rational principal $G$-bundle, where $\rho:U\rightarrow V$ is the
restriction of $\varphi$.
We say that
a $\tilde G$-morphism $\varphi:X\rightarrow Y$ is a $\tilde G$-enriched $n$-almost
principal $G$-bundle if it is so with respect to $U$ and $V$ for some $U$ and $V$.

\begin{lemma}\label{Ref-equiv.lem}
  Let $\varphi:X\rightarrow Y$ be a $\tilde G$-enriched almost principal $G$-bundle between
  $\tilde G$-schemes of finite over $k$.
  Assume that $X$ is normal, and that $\O_Y\rightarrow(\varphi_*\O_X)^G$ is an isomorphism.
  Let $\Refl(\tilde G,X)$ be the category of 
  coherent $(\tilde G,\O_X)$-modules which are reflexive as $\O_X$-modules,
  and
  Let $\Refl(H,Y)$ be the category of 
  coherent $(H,\O_Y)$-modules which are reflexive as $\O_Y$-modules.
  Then we have:
  The functor $\Cal G:\Refl(\tilde G,X)\rightarrow\Refl(H,Y)$ given by $\Cal G(\Cal M)=(\varphi_*\Cal M)^G$
    is an equivalence, and $\Cal F:\Refl(H,Y)\rightarrow\Refl(\tilde G,X)$ given by $\Cal F(\Cal N)=(\varphi^*\Cal N)^{**}$
    is its quasi-inverse, where $(-)^{**}$ is the double dual.
\end{lemma}

\begin{proof}
  Follows immediately from \cite[(10.13),(11.3)]{Has4}.
  A short and self-contained proof for the case that everything is affine can be found in \cite[(2.4)]{HK}.
\end{proof}
    
\paragraph
Let $X$ be a $\tilde G$-scheme, and $L$ be a $\tilde G$-module.
Let $h_X:X\rightarrow\Spec k$ be the structure map.
Then for a quasi-coherent $(\tilde G,\O_X)$-module $\Cal M$, we denote
$\Cal M\otimes_{\O_X}h_X^*L$ by $\Cal M\otimes_k L$.
Note that $G$ is a normal closed subgroup scheme of $\tilde G$.
So $\tilde G$ acts on $G$ by the adjoint action.
We denote this scheme by $G\ad$.
Let $e:\Spec k\rightarrow G\ad$ be the unit element.
It is a $\tilde G$-stable closed immersion.
We denote the $\tilde G$-module $H^0(\Spec k,e^*(\omega_{G\ad}))$ by $\lambda_{\tilde G,G}$,
and we call it the {\em Knop character} of $G$ (enriched by $\tilde G$).
  If $G$ is $k$-smooth, then $\lambda_{G,G}\cong\det_{\frak g}^*$
  by \cite[(28.11)]{Has2}, where $\frak g=\Lie G$ is the adjoint
  representation of $G$, and $\det$ denotes the top exterior power.
  Its dual $\det_{\frak g}=\lambda_{G,G}^*$ is denoted by $\lambda\ad$ in \cite{Knop}, and played an important role in
  studying Gorenstein property of invariant subrings \cite[Satz~2]{Knop}.

\begin{lemma}\label{Watanabe.lem}
  Let $\varphi:X=\Spec T\rightarrow Y=\Spec B$ be a $\tilde G$-enriched almost principal
  $G$-bundle which is also a morphism in $\Cal F(\tilde G)$ with $X$ and $Y$ affine normal.
  Then $\omega_T=(T\otimes_B\omega_B)^{**}\otimes_k\lambda_{\tilde G,G}$,
  where $(-)^{**}=\Hom_T(\Hom_T(-,T),T)$
  is the double dual.
  We also have that $\omega_B=(\omega_T\otimes_k\lambda_{\tilde G,G}^*)^G$.
  In particular, if moreover,  $\lambda_{\tilde G,G}\cong k$, then
  $\omega_T\cong (T\otimes_B\omega_B)^{**}$ and $\omega_B\cong \omega_T^G$.
\end{lemma}

\begin{proof}
  This is a special case of \cite[(11.22)]{Has4}.
\end{proof}

\paragraph
Let $\Lambda$ be a finite-dimensional $k$-algebra.
We say that $\Lambda$ is {\em Frobenius} if ${}_\Lambda\Lambda\cong
D(\Lambda_\Lambda)$ as left $\Lambda$-modules, where $D=\Hom_k(-,k)$.
This is equivalent to say that there is a
nondegenerate bilinear form $\beta
:\Lambda\times \Lambda\rightarrow k$
such that $\beta(ac,b)=\beta(a,cb)$.
If, moreover, we can take such a $\beta$ to be symmetric, we say that
$\Lambda$ is {\em symmetric}.
This is equivalent to say that the bimodule ${}_\Lambda \Lambda_\Lambda$ is
isomorphic to $D({}_\Lambda \Lambda_\Lambda)$.

\paragraph
Let $\Gamma$ be a finite dimensional $k$-Hopf algebra.
We define
\[
\textstyle\int^l_\Gamma:=\{x\in \Gamma^*\mid\forall y\in \Gamma^*\;yx=\epsilon(y)x\},
\]
where $\epsilon:\Gamma^*\rightarrow k$ is given by $\epsilon(y)=y(1_\Gamma)$.
An element of $\int^l_\Gamma$ is called a {\em left integral} on $\Gamma$ (or {\em in} $\Gamma^*$,
according to the terminology in \cite{Montgomery}).

\paragraph
Note that
\[
\textstyle\int^l_\Gamma=(\Gamma^*)^\Gamma=\{\psi\in\Gamma^*\mid\omega_{\gamma^*}(\psi)=\psi\otimes 1\}
=\Hom_\Gamma(\Gamma,k).
\]
Indeed, if $\gamma_1,\ldots,\gamma_n$ is a $k$-basis of $\Gamma$ and $\gamma^*_1,\ldots,\gamma^*_n$ is
its dual basis, then the comodule structure $\omega_{\Gamma^*}$ of $\Gamma^*$ is
given by $\omega_{\Gamma^*}(\alpha)=\sum_{i=1}^n\gamma^*_i\alpha\otimes\gamma_i$ for $\alpha\in\Gamma^*$.
In other words, $\omega(\alpha)=\sum_{(\alpha)}\alpha_{(0)}\otimes\alpha_{(1)}$ is given by
$\sum_{(\alpha)}\langle \beta,\alpha_{(1)}\rangle\alpha_{(0)}=\beta\alpha$ for $\beta\in\Gamma^*$.
So $\omega(\psi)=\psi\otimes 1$ is equivalent to say that $\rho\psi=\epsilon(\rho)\psi$ for $\rho
\in\Gamma^*$, as desired.

\paragraph
We also define $\int^r_\Gamma=\{x\in \Gamma^*\mid \forall y\in \Gamma^*\; xy=\varepsilon(y)x\}$,
and an element of $\int^r_\Gamma$ is called a {\em right integral} on $\Gamma$.
It is known that $\dim_k \int^l_\Gamma=\dim_k \int^r_\Gamma =1$ \cite[Corollary~5.1.6]{Sweedler}.
If $\int^l_\Gamma=\int^r_\Gamma$, then we say that $\Gamma^*$ is unimodular.
Radford proved \cite{Radford} that $\Gamma^*$ is a symmetric algebra if and only if $\Gamma^*$ is unimodular
and $\Cal S^2$ is an inner automorphism of $\Gamma^*$, where $\Cal S$ is the antipode of $\Gamma$.
Suzuki \cite{Suzuki} constructed an example of a finite dimensional unimodular $k$-Hopf algebra which is
not symmetric.

\begin{lemma}\label{Radford.lem}
  Let $G$ be a finite $k$-group scheme, $H$ an affine $k$-group scheme of finite type,
  and $\tilde G=G\times H$ is the direct product.
  Let $\Gamma=k[G]$ be the coordinate ring of $G$.
  Then the following are equivalent.
  \begin{enumerate}
  \item[\rm(1)] $\Gamma^*$ is symmetric.
  \item[\rm(2)] $\Gamma^*$ is unimodular.
  \item[\rm(3)] $\lambda_{\tilde G,G}\cong k$ as $\tilde G$-modules.
  \item[\rm(4)] $\lambda_{G,G}\cong k$ as $G$-modules.
  \end{enumerate}
\end{lemma}

\begin{proof}
  As $\Gamma^*$ is cocommutative, $s^2=\id_{\Gamma^*}$, where $s$ is the antipode of $\Gamma^*$.
  By \cite[Theorem~1,2]{Humphreys} (see also \cite{Radford}), (1)$\Leftrightarrow$(2) holds.

  We prove (2)$\Rightarrow$(3).
  Note that $\Gamma^*$ is a $(G,k[G])$-module.
  In other words, $\Gamma^*$ is a $\Gamma$-Hopf module.
  Let $\zeta:\Gamma^*\rightarrow \Gamma^*\otimes_k \Gamma$ be the map given by
  $\zeta(\gamma)=\sum_{(\gamma)}\gamma_{(0)}(\Cal S \gamma_{(1)})\otimes\gamma_{(2)}$,
  where $\Cal S$ is the antipode of $\Gamma$.
  By the proof of \cite[Theorem~4.1.1]{Sweedler}, $\zeta$ is injective and $\Image\zeta
  =\int^l_\Gamma\otimes k[G]$.
  Now let us consider the same map $\zeta:k[G\ad]^*\rightarrow k[G\ad]^*\otimes k[G\ad]$.
  It is easy to see that this is a $(G,k[G\ad])$-homomorphism.
  Note also that $\int^l_\Gamma$ is a $G$-submodule of both $k[G\ad]$ and $k[G_r]$,
  where $G_r$ is the right regular action.
  As $\int^l_\Gamma=\int^r_\Gamma$, we have that $\int^l_\Gamma$ as the submodule of $k[G\ad]$ is
  also $G$-trivial (isomorphic to $k$).
  Hence $\zeta$ induces an isomorphism $k[G\ad]^*\cong k[G\ad]$ of $(G,k[G\ad])$-modules.
  As $H$ acts trivially on $G\ad$, the isomorphism is that of $(\tilde G,k[G\ad])$-modules.
  Pulling back this isomorphism by the unit element $e:\Spec k\rightarrow G\ad$, we get
  $\lambda_{\tilde G,G}\cong k$, as we have $k[G\ad]^*\cong \omega_{k[G\ad]}$ by the duality of finite morphisms,
  see \cite[(27.8)]{Has2}.

  (3)$\Rightarrow$(4) is trivial.

  (4)$\Rightarrow$(2).
  The argument above shows that $k[G\ad]^*\cong \lambda'\otimes k[G\ad]$,
  where $\lambda'$ is $\int^l_\Gamma$ as a $G$-submodule of $k[G\ad]^*$.
  As $\int^l_\Gamma$ is trivial as a $G$-submodule of $k[G_l]$, where $G_l$ is $G$ with the left regular action,
  we have that
  $\lambda'$ agrees with $\int^l_\Gamma$ as a $G$-submodule of $k[G_r]$.
  The assumption (4) means $k\cong \lambda\cong \lambda'$.
  So $\int^l_\Gamma\subset \int^r_\Gamma$.
  As we know that both $\int^l_\Gamma$ and $\int^r_\Gamma$ are one-dimensional, we
  have that $\Gamma^*$ is unimodular.
\end{proof}

\section{Main theorem}

\paragraph
Let 
$S$ be a $k$-algebra of finite type on which $G$ acts.
Let $A=S^G$ be the ring of invariants.
If the canonical map $\Spec S\rightarrow \Spec A$ is an
almost principal $G$-bundle, then we say that the $G$-action on $S$ is small.
If $V$ is a $G$-module and $S=k[V]=\Sym V^*$ is small, then we say that
the representation $V$ of $G$ is small.
If $G$ is a finite (constant) group, then $V$ is small if and only if the action is faithful,
and $G\subset \GL(V)$ does not have a pseudo-reflection.
Letting each element of $V^*$ of degree one, $S=\Sym V^*$ is a graded $G$-algebra.
So letting $H=\Bbb G_m$ and $\tilde G=G\times H$, we have that $S$ is a $\tilde G$-algebra.

\begin{lemma}[{cf.~\cite[Remark~11.21]{Has4}}]
  In the following cases, we have that the Knop character $\lambda_{\tilde G,G}$ is trivial as $\tilde G$-modules.
  \begin{enumerate}
  \item[\rm(1)] $G$ is finite, and $k[G]^*$ is a symmetric algebra;
  \item[\rm(2)] $G$ is finite and \'etale;
  \item[\rm(3)] $G$ is finite and constant;
  \item[\rm(4)] $G$ is smooth and connected reductive;
  \item[\rm(5)] $G$ is abelian;
  \item[\rm(6)] $G$ is finite, and the identity component $G^\circ$ of $G$ is linearly reductive;
  \item[\rm(7)] $G$ is finite and linearly reductive.
  \end{enumerate}
\end{lemma}

\begin{proof}
  By Lemma~\ref{Radford.lem}, (1) is already proved, and it suffices to show that $\lambda_{G,G}\cong k$
  for (2)--(7).

  For the case that (2), (3), or (4) is assumed, $G$ is $k$-smooth, and hence $\lambda_{G,G}=\det_{\frak g^*}=\bigwedge^{\mathrm{top}}\frak g^*$.
  As the $0^{\mathrm{th}}$ exterior power is always trivial, (4) has been proved.
  The assertion (3) is a special case of (2) (also, direct proofs are well-known, see for example,
  \cite[Example~IV.2.6]{SY}).

  We prove (4).
  We may assume that $k$ is algebraically closed.
  Let $T$ be a maximal torus of $G$.
  As $\lambda_{G,G}=\det_{\frak g}^*$ is one-dimensional,
  it suffices to show that $\det_{\frak g}$ is trivial as a $T$-module.
  We have $\frak g=\frak g_0\oplus \bigoplus_{\alpha\in\Phi}\frak g_\alpha$,
  where $\Phi$ is the set of roots (that is, nonzero weights of $\frak g$).
  It is known that $\dim\frak g_\alpha=1$ for each $\alpha\in\Phi$, and $\Phi=-\Phi$.
  Thus $\det_{\frak g}$ has weight $0+\sum_{\alpha\in\Phi}\alpha=0$.

  We prove (5).
  If $G$ is abelian, then the action of $G$ on $G\ad$ is trivial.
  So $\omega_{G\ad}$ and $\lambda_{G,G}$ are $G$-trivial.

  We prove (6).
  We may assume that $k$ is algebraically closed of characteristic $p>0$.
  By \cite[(3.11)]{Sweedler2}, $G^\circ=\Spec kM$ for some abelian $p$-group $M$.
  Note that $\pi^0(G)=G\red$ is a closed subgroup scheme, and is a constant finite group.
  Note that $G=G^\circ \rtimes G\red$ is a semidirect product.
  As $G\red$ acts on $G^\circ$ by the adjoint action, it acts on the character group $\chi(G^\circ)=M$.
  As the action is that of groups, it fixes the unit element $1_M$ of $M$.
  As a $G^\circ$-module, $k[G^\circ_r]=k[G^\circ_l]$ is decomposed into the sum of one-dimensional
  $G^\circ$-modules as $\bigoplus_{m\in M}k\cdot m$.
  Note that $k\cdot m$ is isomorphic to $k$ if and only if $m=1_M$, and that $\int^r_{k[G^\circ]}=\int^l_{k[G^\circ]}$
  is generated by the projection $\pi:k[G]\rightarrow k$ given by $\pi(1_M)=1$ and $\pi(m)=0$ for $m\in M\setminus\{1_M\}$.
  As $gm\neq 1_M$ if $m\neq 1_M$ and $g1_M=1_M$ for any $g\in G\red$, we have that $g\pi=\pi$ for any $g\in G\red$,
  where $(g\pi)(m)=\pi(g^{-1}(m))$.
  This shows that $\lambda_{G,G^\circ}\cong k$.
  As $G^\circ$ is a $G$-stable (closed and) open neighborhood of the unit element $e$ in $G\ad$,
  we have that $\lambda_{G,G}=\lambda_{G,G^\circ}\cong k$, as desired.

  We prove (7).
  By \cite[Lemma~2.2]{Has3}, $G^\circ$ is linearly reductive.
  By (6), the assertion is clear now.
\end{proof}

\begin{example}[{cf.~\cite[p.~51]{Knop}}]
  $\lambda_{G,G}$ is not $G$-trivial in general, even if $k$ is an algebraically closed field of
  characteristic zero, and $G$ is $k$-smooth.
  Let $k=\Bbb C$, and consider
  \[
  G=O_2=\{A\in\GL_2(\Bbb C)\mid {}^tAA=E_2\},
  \]
  where $E_2$ is the identity matrix.
  Then the Lie algebra $\frak g$ of $G$ is
  \[
  \{B\in\gl_2(\Bbb C)=\Mat_2(\Bbb C)\mid {}^tB+B=O\},
  \]
  on which $G$ acts by the action $(A,B)\mapsto AB{}^tA$.
  It is easy to see that the action is nontrivial, and hence $\lambda_{G,G}=\frak g^*$ is also nontrivial.
\end{example}

\begin{example}
  $\lambda_{G,G}$ is not $G$-trivial in general, even if $G$ is finite.
  Consider the restricted Lie algebra (see \cite[(V.7)]{Jacobson} for definition)
  $L$ over a field $k$ of characteristic $p>0$
  with the basis $e,f$ with the relations $[f,e]=e$, $f^p=f$, and $e^p=0$.
  Take the restricted universal enveloping algebra $V$ of $L$.
  Letting each element of $x\in L$ primitive (i.e., $\Delta(x)=x\otimes 1+1\otimes x$),
  $V$ is a $p^2$-dimensional cocommutative Hopf algebra which is not unimodular,
  see \cite[p.~85]{LS}.
  Letting $G=\Spec V^*$, we have that $\lambda_{G,G}$ is not trivial by Lemma~\ref{Radford.lem}.
\end{example}

\begin{lemma}\label{pre-q-Gor.lem}
  Let $k$ be a field, $G$ and $H$ be affine $k$-group schemes of finite type, and
  $\tilde G=G\times H$.
  Let $S$ be a $\tilde G$-algebra, and assume that the action of $G$ on $S$ is small.
  We assume that $S$ is normal, and $\lambda_{G,G}$ is trivial.
  Let $L$ be an $(H,A)$-module which is projective as an $A$-module.
  Then the following are equivalent:
  \begin{enumerate}
  \item[\rm(1)] $\omega_S\cong S\otimes_A L$ as $(\tilde G,S)$-modules;
  \item[\rm(2)] $\omega_A\cong L$ as $(H,A)$-modules,
  \end{enumerate}
  where the action of $G$ on $L$ is trivial.
\end{lemma}

\begin{proof}
  (1)$\Rightarrow$(2).
  By Lemma~\ref{Watanabe.lem},
  $\omega_A\cong\omega_S^G\cong (S\otimes_AL)^G\cong A\otimes_AL\cong L$,
  since $L$ is $G$-trivial and $A$-flat.

  (2)$\Rightarrow$(1).
  We have $\omega_S^G\cong\omega_A\cong L$.
  Applying the functor $(S\otimes_A-)^{**}$, which is the quasi-inverse of $(-)^G:\Refl(\tilde G,S)\rightarrow\Refl(H,A)$,
  we get isomorphisms
  \[
  \omega_S\cong (S\otimes_A\omega_S^G)^{**}\cong (S\otimes_AL)^{**}\cong S\otimes_A L
  \]
  of $(\tilde G,S)$-modules.
\end{proof}

\begin{theorem}\label{main.thm}
  Let $k$ be a field, $G$ be an affine $k$-group scheme of finite type, and $V$ be a small $G$-module
  of dimension $n<\infty$.
  We assume that $\lambda_{G,G}$ is trivial.
  Let $H=\Bbb G_m$ be the one-dimensional torus, and let $V$ be of degree one as an $H$-module
  so that $S=\Sym V^*$ is a $\tilde G$-algebra generated by degree one elements,
  where $\tilde G=G\times H$.
  We set $A=S^G$.
  Then we have
  \begin{enumerate}
    \item[\rm(i)] $\omega_A\cong\omega_S^G$ as $(H,A)$-modules;
    \item[\rm(ii)] $a(A)\leq -n$ in general, where $a(A)$ denotes the $a$-invariant.
  \end{enumerate}
  Moreover, the following are equivalent:
  \begin{enumerate}
  \item[\rm(1)] The action $G\rightarrow\GL(V)$ factors through $\SL(V)$;
  \item[\rm(2)] $\omega_S\cong S(-n)$ as $(\tilde G,S)$-modules;
  \item[\rm(3)] $\omega_S\cong S$ as $(G,S)$-modules;
  \item[\rm(4)] $\omega_A\cong A(-n)$ as $(H,A)$-modules;
  \item[\rm(5)] $A$ is quasi-Gorenstein;
  \item[\rm(6)] $A$ is quasi-Gorenstein and $a(A)=-n$;
  \item[\rm(7)] $a(A)=-n$.
  \end{enumerate}
\end{theorem}

\begin{proof}
  The assertion (i) is clear by Corollary~\ref{Watanabe.lem}.
  We prove (ii).
  We have an $(H,A)$-linear isomorphism $\omega_A\rightarrow\omega_S^G$ by assumption, and $\omega_S^G
  \subset\omega_S=S\otimes_k \det_V$.
  So $a(A)\leq a(S)=-n$ in general.
  The equality holds only if $\det_V$ is trivial.
  Namely, we have (7)$\Rightarrow$(1).
  
  (1) is equivalent to say that $\det_V\cong k(-n)$.
  Combining this with the fact $\omega_S\cong S\otimes_k \det_V$, we get (1)$\Rightarrow$(2).

  (2)$\Rightarrow$(3) is trivial.

  (3)$\Rightarrow$(1).
  $S\otimes_k\det_V\cong \omega_S\cong S$ as $(G,S)$-modules.
  So
  \[
  \textstyle\det_V\cong S/S_+\otimes_S(S\otimes_k\det_V)\cong S/S_+\otimes_S\omega_S\cong S/S_+\otimes_SS\cong S/S_+\cong k
  \]
  as $G$-modules.
  This shows (1).

  (2)$\Rightarrow$(4).
   $\omega_A\cong\omega_S^G\cong S(-n)^G\cong A(-n)$ as $(H,A)$-modules.

  (4)$\Rightarrow$(6)$\Rightarrow$(5) is trivial.

  (5)$\Rightarrow$(3).
  By assumption, $\omega_A$ is projective.
  As $A$ is positively graded and $\omega_A$ is a graded finitely generated module of rank one, we have that
  $\omega_A\cong A(a)$ for some $a\in\ZZ$.
  By Lemma~\ref{pre-q-Gor.lem}, we have that $\omega_S\cong S\otimes_AA(a)\cong S(a)$ as $(\tilde G,S)$-modules.
  Forgetting the grading, we have that $\omega_S\cong S$ as $(G,S)$-modules, as desired.

  (6)$\Rightarrow$(7) is trivial.
\end{proof}

\begin{remark}
  Goel--Jeffries--Singh \cite{GJS} proved better theorems than
  Theorem~\ref{main.thm} for the case that $G$ is finite and constant.
  They proved the inequality $a(A)\leq -n$ without
  assuming that the action is small.
  They also prove there that $a(A)=-n$ implies that the action is small (and hence $G\subset\SL(V)$),
  see \cite[Proposition~4.1, Theorem~4.4]{GJS}.
  The author does not know if these are true for a general finite group scheme $G$.

  The equivalence (1)$\Leftrightarrow$(5) for the case that $G$ is finite and constant
  was first proved by Fleischmann and Woodcock \cite{FW} and Braun \cite{Braun}.
  The author proved that $\omega_S^G\cong \omega_A$ if $G$ is finite linearly reductive, without assuming that
  the action is small \cite[(32.4)]{Has2}.
  The equivalence (1)$\Leftrightarrow$(5) for the case that $G$ is finite linearly reductive
  was proved by Liedtke--Yasuda \cite[Proposition~4.7]{LY} 
  ($A$ is strongly $F$-regular this case, and hence quasi-Gorenstein is equivalent to Gorenstein there).
\end{remark}

\begin{example}
  We give an example of higher-dimensional $G$.
  Let $m$, $n$ and $t$ be positive integers such that $2\leq t\leq m\leq n$.
  Let $W_1=k^n$, $W_2=k^m$, $E=k^{t-1}$, and $G=\GL(E)$.
  We consider that $G$ acts on $E$ as a vector representation, while the actions of $G$ on $W_1$ and $W_2$
  are trivial.
  We set $V=\Hom_k(E,W_2)\oplus \Hom_k(W_1,E)$, $S=\Sym V^*$, and $A=S^G$.
  We define $X=V=\Hom(E,W_2)\times\Hom(W_1,E)=\Spec S=E^n\times (E^*)^m$, and $Y=\Spec A=X/\!/G$.
  The quotient map $\pi:X\rightarrow Y$ is identified with the map
  $\Pi:X\rightarrow Y_t$ given by $(\varphi,\psi)\mapsto\varphi\circ\psi$,
  where $Y_t=\{\rho\in\Hom(W_1,W_2)\mid\rank\rho<t\}$ is the determinantal variety,
  see \cite{DP}.
  Note that $\Pi$ is a $\GL(W_1)\times G\times\GL(W_2)$-enriched almost principal $G$-bundle,
  see \cite{Has1}.
  So by the theorem, we have that $a(A)\leq a(S)=-(m+n)(t-1)$,
  and the equality holds if and only if $A$ is Gorenstein.
  Note also that the usual grading of $A=k[\Hom(W_1,W_2)^*]/I_t$, where $I_t$ is the determinantal ideal,
  is the one such that each element of $\Hom(W_1,W_2)^*$ is of degree one.
   However, the grading used here is the one which is inherited from the grading of $S$,
    and each element of $\Hom(W_1,W_2)^*$ is of degree two.
    For the case that $k$ is of characteristic zero, Lascoux's resolution \cite{Lascoux} tells us that
    $a(A)=2(-mn+n(m-t+1))=-2n(t-1)\leq -(m+n)(t-1)=a(S)$, doubling the degree to adopt our grading inherited from $S$.
    Being a graded ASL over a distributive lattice, $A$ is Cohen--Macaulay, and the Hilbert series of $A$ is independent of $k$, see \cite{BH}.
    So $a(A)$ is also independent of $k$, and we always have $a(A)=-2n(t-1)$.
    So $a(A)=a(S)$ if and only if $m=n$.
    This shows that $A$ is Gorenstein if and only if $m=n$,
    and this is the well-known theorem by Svanes \cite{Svanes}.
\end{example}

\begin{example}
  Let $k$ be an algebraically closed field of characteristic $p>0$, and $\ell$ be a prime number
  which does not divide $p(p-1)$.
  In particular, $\ell$ is odd.
  Let
  \[
  G=\left.\left\{
  \begin{bmatrix}
    t & \alpha \\
    0 & 1
  \end{bmatrix}\,
  \right |\,
  t\in\boldsymbol\mu_\ell,\;
  \alpha\in\boldsymbol\alpha_p
  \right\},
  \]
  where $\boldsymbol{\alpha}_p=\Spec k[a]/(a^p)$ is the first Frobenius kernel of the additive group $\Bbb G_a=\Spec k[a]=\Bbb A^1$,
  and $\boldsymbol\mu_\ell=\Spec k[T]/(T^\ell-1)\subset\GL_1$.
  Note that $G$ acts on the vector representation $W=k^2$ in a natural way.
  Let $V=W\oplus W^*$.
  It is easy to see that $V$ is small and $G\subset\SL(V)$.
  It is also easy to see that $\lambda_{G,G}=\soc k[(\boldsymbol{\alpha}_p)\ad]^*=(\soc k[a]/(a^p))^*=(k a^{p-1})^*$.
  With the adjoint action, we have $\sigma_t\cdot a= t^{-2}a$, where $\sigma=\begin{bmatrix} t & 0 \\ 0 & 1
  \end{bmatrix}$.
  By assumption, $\lambda_{G,G}$ is nontrivial.
  By \cite[Corollary~11.22]{Has4}, $A=S^G$ is not quasi-Gorenstein, where $A=k[V]=\Sym V^*$.
\end{example}

\paragraph Let $G$ be a finite $k$-group scheme, $H$ a $k$-group scheme of finite type, and set $\tilde G=G\times H$.
Let $S$ be a $\tilde G$-algebra, and $A=S^G$.
In \cite{CR}, the trace map $\Tr_{S/A}:S\rightarrow A$ is defined.
Let $\delta_G:k[G]\rightarrow k$ be a non-zero left integral (that is, $\delta_G\in\int^l_{k[G]}\setminus\{0\}$).
This is equivalent to say that $\delta_G\in\Hom_G(k[G_l],k)\setminus\{0\}$.
For any $G$-algebra $S$, let $\Tr_{S/A}:S\rightarrow S'$ be the composite
\[
S\xrightarrow{\omega_S}S'\otimes_k k[G]\xrightarrow{1_{S'}\otimes \delta_G}S'\otimes_k k=S',
\]
where $S'$ is the $A$-module $S$ with the trivial $G$-action.
By \cite[Definition-Proposition~3.6]{CR}, the image of $\Tr_{S/A}$ is contained in $A=S^G$, and hence
the map $\Tr_{S/A}:S\rightarrow A$ is induced.
It is easy to see that $\Tr_{S/A}$ is $A$-linear.

\paragraph
Assume that the Hopf algebra $k[G]^*$ is unimodular.
Then $\delta_G$ is also a right integral.
That is, $\delta_G:k[G_r]\rightarrow k$ is $G$-linear (note that $k[G_r]$ is a $k[G]$-comodule algebra
letting the coproduct $\Delta:k[G_r]\rightarrow k[G_r]\otimes k[G_r]$ the coaction).
As $\omega_S:S\rightarrow S'\otimes_k k[G_r]$ is also $G$-linear, we have that
$\Tr_{S/A}:S\rightarrow A$ is $(G,A)$-linear.
Moreover, $\delta_G$ is $H$-linear, since $H$ acts trivially on $G_r$.
Letting the action of $H$ on $S'$ be the same as that on $S$, we have that $\omega_S:S\rightarrow S'\otimes_k k[G_r]$ is
also $H$-linear.
Thus $\Tr_{S/A}$ is $(\tilde G,A)$-linear.

\begin{theorem}\label{CR.thm}
  Let $S$ be a $k$-algebra of finite type.
  Let $G$ be a finite $k$-group scheme, $H$ be a $k$-group scheme of finite type, and $\tilde G=G\times H$.
  Assume that $\tilde G$ acts on $S$.
  If $k[G]^*$ is symmetric and either
  \begin{enumerate}
  \item[\rm(1)] The map $\Spec S\rightarrow \Spec A$ is a $\tilde G$-enriched principal $G$-bundle; or
    \item[\rm(2)] The action of $G$ on $S$ is small, and $S$ satisfies the $(S_2)$-condition,
  \end{enumerate}
    then $\zeta:S\rightarrow \Hom_A(S,A)$ \($s\mapsto (t\mapsto \Tr_{S/A}(st))$\) is an isomorphism of $(\tilde G,S)$-modules.
\end{theorem}

\begin{proof}
  This is \cite[Corollary~3.13]{CR} except that we need to prove that the map $\zeta$ is $\tilde G$-linear.
  This is done in the discussion above.
\end{proof}

\begin{flushleft}
Mitsuyasu Hashimoto\\
Department of Mathematics\\
Osaka Metropolitan University\\
Sumiyoshi-ku, Osaka 558--8585, JAPAN\\
e-mail: {\tt mh7@omu.ac.jp}
\end{flushleft}


\begin{thebibliography}{Wat2}
\def\ji#1#2(#3)#4-#5.{\newblock{\em#1} {\bf#2} (#3), #4--#5.}
\def\GTM#1{Graduate Texts in Math. {\bf #1}, Springer}
\def\SLN#1{Lecture Notes in Math. {\bf #1}, Springer}

\bibitem[Bra]{Braun}
  A. Braun,
  On the Gorenstein property for modular invariants,
  {\em J. Algebra} {\bf 345} (2011), 81--99.

\bibitem[BH]{BH}
  W. Bruns and J. Herzog,
  {\em Cohen--Macaulay Rings. 2nd ed.},
  {\em Cambridge Studies in Advanced Mathematics} {\bf 39},
  Cambridge University Press, 1998.

\bibitem[C-R]{CR}
  J. Carvajal-Rojas,
  Finite torsors over strongly $F$-regular singularities,
  {\em \'Epijournal G\'eom. Alg\'ebrique} {\bf 6} (2022), Art. 1, 30pp.

\bibitem[DP]{DP}
  C. de Concini and C. Procesi,
  A characteristic free approach to invariant theory,
  {\em Adv. Math.} {\bf 21} (1976), 330--354.

\bibitem[FW]{FW}
  P. Fleischmann and C. Woodcock,
  Relative invariants, ideal classes and quasi-canonical modules
  of modular rings of invariants,
  {\em J. Algebra} {\bf 348} (2011), 110--134.

\bibitem[GJS]{GJS}
  K. Goel, J. Jeffries, and A. Singh,
  Local cohomology of modular invariants,
  {\tt arXiv:2306.14279v1}

\bibitem[Has1]{Has1}
  M. Hashimoto,
  Another proof of theorems of De Concini and Procesi,
  {\em J. Math. Kyoto Univ.} {\bf 45} (2005), 701--710.

\bibitem[Has2]{Has2}
M. Hashimoto,
Equivariant twisted inverses,
{\em Foundations of Grothendieck Duality for Diagrams
of Schemes} (J.~Lipman, M.~Hashimoto),
\SLN{1960} (2009), pp.~261--478.

\bibitem[Has3]{Has3}
  M. Hashimoto,
  Classification of the linearly reductive finite subgroup schemes of $\SL_2$,
  {\em Acta Math. Vietnam.} {\bf 40} (2015), 527--534.

\bibitem[Has4]{Has4}
  M. Hashimoto,
  Equivariant class group. III.
  Almost principal fiber bundles,
  {\tt arXiv:1503.02133}

\bibitem[HK]{HK}
  M. Hashimoto and F. Kobayashi,
  Generalized $F$-signatures of the rings of invariants of finite group schemes,
  {\tt arXiv:2304.12138v3}

\bibitem[Hum]{Humphreys}
  J. E. Humphreys,
  Symmetry for finite dimensional Hopf algebras,
  {\em Proc. Amer. Math. Soc.} {\bf 68} (1978), 143--146.

\bibitem[Jac]{Jacobson}
  N. Jacobson,
  {\em Lie Algebras}, Dover, 1979.
  
\bibitem[Kno]{Knop}
  F. Knop,
  Der kanonische Modul eines Invariantenrings,
  {\em J. Algebra} {\bf 127} (1989), 40--54.

\bibitem[LS]{LS}
  R. G. Larson and M. E. Sweedler,
  An associative orthogonal bilinear form for Hopf algebras,
  {\em Amer. J. Math.} {\bf 91} (1969), 75--94.

\bibitem[Las]{Lascoux}
  A. Lascoux,
  Syzygies des vari\'et\'es d\'eterminantales,
  {\em Adv. Math.} {\bf 30} (1978), 202--237.

\bibitem[LY]{LY}
  C. Liedtke and T. Yasuda,
  Non-commutative resolutions of linearly reductive quotient singularities,
  {\tt arXiv:2304.14711v2}

\bibitem[Mon]{Montgomery}
  S. Montgomery,
  {\em Hopf Algebras and Their Actions on Rings,}
  {\em CBMS Regional Conf. Ser. in Math.} {\bf 82},
  AMS, 1993.


\bibitem[Rad]{Radford}
  D. E. Radford,
  The trace function and Hopf algebras,
  {\em J. Algebra} {\bf 163} (1994), 583--622.

\bibitem[SY]{SY}
  A. Skowro\'nski and K. Yamagata,
  {\em Frobenius Algebras.  I. Basic Representation Theory},
  European Mathematical Society, 2012.

\bibitem[Suz]{Suzuki}
  S. Suzuki,
  Unimodularity of finite dimensional Hopf algebras,
  {\em Tsukuba J. Math.} {\bf 20} (1996), 231--238.

\bibitem[Sva]{Svanes}
  T. Svanes, Coherent cohomology on Schubert subschemes of flag schemes and applications,
  {\em Adv. Math.} {\bf 14} (1974), 369--453.
  
\bibitem[Swe1]{Sweedler}
  M. E. Sweedler,
  {\em Hopf Algebras}, Benjamin, 1969.

\bibitem[Swe2]{Sweedler2}
  M. E. Sweedler,
  Connected fully reducible affine group schemes in positive characteristic are Abelian,
  {\em J. Math. Kyoto Univ.} {\bf 11-1} (1971), 51--70.
  
\bibitem[Wat1]{Wat1}
  K. Watanabe, Certain invariant subrings are Gorenstein. I.
  {\em Osaka Math. J.} {\bf 11} (1974), 1--8.

\bibitem[Wat2]{Wat2}
  K. Watanabe, Certain invariant subrings are Gorenstein. II.
  {\em Osaka Math. J.} {\bf 11} (1974), 379--388.

\end{thebibliography}
\end{document}